\documentclass{amsart}

\usepackage[utf8]{inputenc}

\usepackage{enumitem}

\newtheorem{thm}{Theorem}[section]
\theoremstyle{definition}
\newtheorem{dfn}[thm]{Definition}
\theoremstyle{plain}
\newtheorem{lem}[thm]{Lemma}
\theoremstyle{plain}

\theoremstyle{definition}
\newtheorem{conj}[thm]{Conjecture}
\theoremstyle{definition}

\usepackage{reversemath}

\newcommand{\Uf}{\mathcal{U}}
\newcommand{\KS}{\mathrm{KS}}
\newcommand{\FPT}{\mathrm{FPT}}

\usepackage[pagebackref]{hyperref}
\hypersetup{%
    colorlinks=true,
    linkcolor=red,
    citecolor=blue,
    urlcolor=blue
}

\title[{Arrow}'s theorem, ultrafilters, and reverse mathematics]%
    {{Arrow}'s theorem, ultrafilters, \\ and reverse mathematics}
\author{Benedict Eastaugh}
\address{%
    Department of Philosophy \\
    University of Warwick \\
    Coventry CV4 7AL \\
    UK
}
\email{benedict@eastaugh.net}
\date{26 February 2024}

\begin{document}

\begin{abstract}
    This paper initiates the reverse mathematics of social choice theory,
    studying Arrow's impossibility theorem and related results including
    Fishburn's possibility theorem and the Kirman--Sondermann theorem within the
    framework of reverse mathematics. We formalise fundamental notions of social
    choice theory in second-order arithmetic, yielding a definition of countable
    society which is tractable in $\RCA_0$. We then show that the
    Kirman--Sondermann analysis of social welfare functions can be carried out
    in $\RCA_0$. This approach yields a proof of Arrow's theorem in $\RCA_0$,
    and thus in $\PRA$, since Arrow's theorem can be formalised as a $\Pi^0_1$
    sentence. Finally we show that Fishburn's possibility theorem for countable
    societies is equivalent to $\ACA_0$ over $\RCA_0$.
\end{abstract}

\maketitle


\section{Introduction}
\label{sec:intro}

Arrow's 1950 impossibility theorem \cite{Arrow1950, Arrow1963} is a foundational
result in social choice theory. If a society contains only finitely many voters,
then any aggregation of individual preference orderings (called a \emph{social
welfare function}) respecting Arrow's conditions of unanimity and independence
of irrelevant alternatives is dictated by a single voter.
The theorem therefore appears to place substantial limits on the existence of
methods for social decision-making that are fair, rational, and uniform. It has
a wide range of applicability including the problems of selecting candidates in
elections, deciding on public policies, and choosing between rival scientific
theories. As such it has exerted a substantial influence on economics
\cite{Sen1970a, Hammond1976}, political science \cite{Riker1982}, and philosophy
\cite{Hurley1985, Okasha2011a}.

Although Arrow's theorem is essentially a result in finitary combinatorics,
later developments in social choice theory in the 1970s brought in more powerful
methods such as non-principal ultrafilters, which Fishburn~\cite{Fishburn1970}
used to show that infinite societies have non-dictatorial social welfare
functions. This result and others like it have led mathematical economists to
grapple with non-constructivity and applications of the axiom of choice
\cite{Litak2018}.
However, for economically and philosophically relevant models such as societies
which are countable or continuous, reverse mathematics offers a more appropriate
framework for gauging where (and what) non-constructive set existence axioms are
actually necessary in social choice theory.

This paper initiates the reverse mathematics of social choice theory, studying
Arrow's impossibility theorem and related results including Fishburn's
possibility theorem within the framework of reverse mathematics. By defining
fundamental notions of social choice theory in second-order arithmetic, we show
that an influential analysis of social welfare functions in terms of
ultrafilters by Kirman and Sondermann~\cite{KirmanSondermann1972} can be carried
out in $\RCA_0$. This allows us to establish that Arrow's theorem, when
formalised as a statement of first-order arithmetic, is provable in primitive
recursive arithmetic. Fishburn's possibility theorem, on the other hand, uses
non-constructive resources in an essential way, and we prove that its
restriction to countable societies is equivalent to $\ACA_0$.


In the classical Arrovian framework, a \emph{society} $\mathcal{S}$ consists of
a set $V$ of \emph{voters}, a set $X$ of \emph{alternatives} (or candidates),
together with the set $W$ of all weak orders of $X$ (representing the different
ways in which the set of alternatives can be rationally ordered), a set
$\mathcal{A}$ of \emph{coalitions} of voters, and a set $\mathcal{F}$ of
\emph{profiles}, i.e.\ functions $f : V \to W$ representing different elections
or voting scenarios.
In Arrow's framework, $\mathcal{A}$ and $\mathcal{F}$ satisfy a condition known
as \emph{unrestricted domain} (or \emph{universal domain}), meaning that
$\mathcal{A} = \powset{V}$ and $\mathcal{F} = W^V$, the set of all functions
$f : V \to W$.

Given alternatives $x,y \in X$, a profile $f : V \to W$, and a voter $v \in V$,
we write
\begin{equation*}
    x \lesssim_{f(v)} y
\end{equation*}
to mean that voter $v$ ranks $x$ at least as highly $y$ under profile $f$, and
\begin{equation*}
    x <_{f(v)} y
\end{equation*}
to mean that voter $v$ strictly prefers $x$ to $y$ under profile $f$.
A \emph{social welfare function} $\sigma$ for a society $\mathcal{S}$ maps
profiles in $\mathcal{F}$ to weak orders in $W$, and represents one way of
consistently aggregating individual preference orderings into an overall social
preference ordering. We write
\begin{equation*}
    x \lesssim_{\sigma(f)} y
\end{equation*}
to mean that the social welfare function $\sigma$ ranks $x$ at least as highly
as $y$ under profile $f$, and similarly for $x <_{\sigma(f)} y$.
If $R$ is a weak ordering and $Y \subseteq X$, we write $R\restr{Y}$ to mean
$R \cap Y^2$.
This lets us state Arrow's conditions more precisely.
\begin{enumerate}
    \item Unanimity:
        If $x <_{f(v)} y$ for all $v \in V$, then $x <_{\sigma(f)} y$.
    \item Independence of irrelevant alternatives:
        If $f(v)\restr{\{x,y\}} = g(v)\restr{\{x,y\}}$ for all $v \in V$ then
        $\sigma(f)\restr{\{x,y\}} = \sigma(g)\restr{\{x,y\}}$.
    \item Non-dictatoriality:
        There is no $d \in V$ such that for all $f \in \mathcal{F}$,
        if $x <_{f(d)} y$ then $x <_{\sigma(f)} y$.
\end{enumerate}

\begin{thm}[Arrow's impossibility theorem]
    \label{thm:Arrow}
    Suppose $\mathcal{S} = \str{V,X,\mathcal{A},\mathcal{F}}$ is a society
    satisfying unrestricted domain such that $V$ is a nonempty and finite set of
    voters, and $X$ is a finite set of alternatives with $\lh{X} \geq 3$.
    Then there exists no social welfare function $\sigma : \mathcal{F} \to W$
    satisfying unanimity, independence, and non-dictatoriality.
\end{thm}

Fishburn~\cite{Fishburn1970} offered a way out of Arrow's impossibility result,
showing that Arrow's conditions are consistent if we drop the requirement that
$V$ is finite.%
\footnote{%
The result was apparently already known to Julian Blau in 1960, although Blau
never published his proof \cite[p.~16]{Fishburn1987}. It should therefore
perhaps be called the Blau--Fishburn possibility theorem, as suggested in
\cite[p.~283]{Saari1991}.
}

\begin{thm}[Fishburn's possibility theorem]
    \label{thm:FPT_ZFC}
    Suppose $\mathcal{S} = \str{V,X,\mathcal{A},\mathcal{F}}$ is a society
    satisfying unrestricted domain such that $V$ is an infinite set of voters,
    and $X$ is a finite set of alternatives with $\lh{X} \geq 3$.
    Then there exists a social welfare function $\sigma : \mathcal{F} \to W$
    satisfying unanimity, independence, and non-dictatoriality.
\end{thm}

Infinite societies are widely used in mathematical economics \cite{Aumann1964,
Hildenbrand1970, Hammond1979, HeSunSun2017}.%
\footnote{%
Schmitz~\cite[p.~193]{Schmitz1977} writes that
``measure spaces $(V, \mathcal{V}, \mu)$ of infinitely many agents with
$\mu$-atoms [are] of some interest since these spaces can serve as models for
large economies with preformed coalitions (e.g., religious, regional or social
groups) and/or with powerful companies or political parties''.
See also the introductory discussion of countably infinite societies in
\cite{Mihara1997}, and the references on population ethics for infinite
societies in \S 7.4 of \cite{EaswaranHajekMancosu2021}.
}
Fishburn's theorem is therefore of antecedent interest in its application
domain, despite the prima facie implausibility of infinite `societies'.

On a mathematical level, Fishburn's possibility theorem is best understood in
the context of an influential analysis by Kirman and Sondermann~%
\cite{KirmanSondermann1972} which shows that social welfare functions satisfying
unanimity and independence correspond to ultrafilters.
Arrow had already introduced the notion of a \emph{$\sigma$-decisive coalition}
for a social welfare function $\sigma$: a set $C \subseteq V$ such that if
$x <_{f(v)} y$ for every $v \in C$, then $x <_{\sigma(f)} y$. Kirman and
Sondermann established that the collection of all $\sigma$-decisive coalitions
forms an ultrafilter which is principal if and only if $\sigma$ is dictatorial.

\begin{thm}[Kirman--Sondermann theorem]
    \label{thm:KS_ZF}
    Suppose $\mathcal{S} = \str{V,X,\mathcal{A},\mathcal{F}}$ is a society
    satisfying unrestricted domain such that $V$ is a nonempty set of voters,
    and $X$ is a finite set of alternatives with $\lh{X} \geq 3$. For any social
    welfare function $\sigma : \mathcal{F} \to W$ satisfying unanimity and
    independence, the set
    \begin{equation*}
        \Uf_\sigma = \{ C \in \mathcal{A} : \text{$C$ is $\sigma$-decisive} \},
    \end{equation*}
    forms an ultrafilter on $\mathcal{A}$ which is principal if and only if
    $\sigma$ is dictatorial.
\end{thm}

Arrow's theorem is an immediate consequence of the Kirman--Sondermann theorem:
as every ultrafilter on a finite set is principal and hence generated by a
singleton $\{ d \}$, any social welfare function for a society with a finite set
$V$ of voters must be dictatorial.
The Kirman--Sondermann theorem also provides us with our first reverse
mathematics-style result. Since it is provable in $\ZF$, any non-dictatorial
social welfare function $\sigma$ for a society with an infinite set $V$ of
voters will give rise to a non-principal ultrafilter $\Uf_\sigma$ on
$\powset{V}$.

\begin{thm}
    \label{thm:FPT_equiv_NPU}
    Fishburn's possibility theorem is equivalent over $\ZF$ to the statement
    that for every infinite set $V$ there exists a non-principal ultrafilter on
    $\powset{V}$.
\end{thm}

The existence of non-principal ultrafilters is unprovable in $\ZF$
\cite{Blass1977}, but is (strictly) implied by the axiom of choice
\cite{Jech1973, PincusSolovay1977}.
Many therefore consider Fishburn's possibility theorem to be highly
non-constructive \cite{Mihara1999, BrunnerMihara2000, CampbellKelly2002,
Cato2017}. At least prima facie, this is a substantial problem for any genuine
application of Fishburn's possibility theorem in social choice theory, a field
which is supposed to apply to everyday social decision-making processes such as
national elections or votes in a hiring committee.%
\footnote{%
For a detailed discussion in this vein see \S\S 2--3 of \cite{Litak2018}.
}
This kind of concern with applicability lies behind a wide range of studies of
Arrow's theorem using tools from computability theory and computational
complexity theory. Amongst the former are the work of
Lewis~\cite{Lewis1988} in the 1980s and Mihara~\cite{Mihara1997, Mihara1999} in
the 1990s, while the latter is the preserve of the flourishing field of
computational social choice theory \cite{ChevaleyreEndrissLang2007,
BrandtConitzerEndriss2016}.

Lewis~\cite{Lewis1988} worked principally with a notion of ``recursively
enumerable society'' in which $V = \omega$, the algebra of coalitions
$\mathcal{A}$ is restricted to include only computably enumerable sets, and the
set $\mathcal{F}$ of profiles is restricted to include only computable
functions. The set $X$ of alternatives must be have at least 3 elements, and be
at most countably infinite.
Lewis proved a weak version of Arrow's theorem for such societies, showing that
if $\sigma$ is a computable social welfare function for a recursively enumerable
society $\mathcal{S}$, then for each profile $f \in \mathcal{F}$ there exists a
`dictator' $d$ such that for all $x,y \in X$, if $x <_{f(d)} y$, then
$x <_{\sigma(f)} y$. This `dictator' is not necessarily unique across all
profiles, and hence not a dictator in Arrow's original sense.%
\footnote{%
For a more detailed appraisal of Lewis's framework and results, see appendix~F
of Mihara's dissertation~\cite{Mihara1995}.
}

Mihara's approach in \cite{Mihara1997} is somewhat different, working with a
single society $\mathcal{S}$ in which $V = \omega$, and the coalition algebra
$\mathcal{A}$ is precisely the set $\REC$ of all computable sets. Mihara allows
a broader range of profiles in $\mathcal{F}$, namely those which are measurable
by sets in $\REC$.%
\footnote{%
Measurable profiles are introduced at the start of \S\ref{sec:societies}.
}
The set of alternatives $X$ can be any set with at least 3 elements, although
the computability requirements mean that only countably many alternatives will
actually end up being considered by any given social welfare function.
Unlike Lewis, Mihara defines a dictator as Arrow does: a single individual whose
preferences determine the social ordering across all profiles.
Mihara proves that any computable non-dictatorial social welfare function for
the society based on the coalition algebra $\REC$ must compute $\zeroj$.
The recursive counterexample which we give to Fishburn's possibility theorem at
the end of \S\ref{sec:fpt} improves on Mihara's result by constructing a
countable society which does not contain all computable sets as coalitions, and
can be coded as a single computable set, but all of whose non-dictatorial social
welfare functions compute $\zeroj$.
In~\cite{Mihara1999}, Mihara shows that there exist non-dictatorial social
welfare functions for this society which are computable relative to $\zerojj$.%
\footnote{%
A natural question left open by \cite[p.~270]{Mihara1999} is whether there exist
non-dictatorial social welfare functions for Mihara's society which are
computable relative to $\zeroj$. A generalisation of this question is discussed
at the end of \S\ref{sec:fpt}.
}

The aim of this paper is to provide a more nuanced analysis of the situation
regarding Arrow's theorem, Fishburn's theorem, and their relative
(non-)constructivity in terms of the hierarchy of subsystems of second-order
arithmetic studied in reverse mathematics. After briefly introducing the
relevant background from reverse mathematics and social choice theory in
\S\ref{sec:prelims}, we present a canonical sequence of definitions in
\S\ref{sec:societies} for investigating the proof-theoretic strength of theorems
in social choice theory. This investigation begins with Arrow's impossibility
theorem and Fishburn's possibility theorem, but the framework is sufficiently
general and flexible to accommodate future research on other landmark results in
social choice theory such as the Gibbard--Satterthwaite theorem
\cite{Gibbard1973, Satterthwaite1975}.

The central definition is that of a \emph{countable society}: a structure
$\mathcal{S} = \str{V,X,\mathcal{A},\mathcal{F}}$ in which $V \subseteq \N$, and
the algebra of coalitions $\mathcal{A} \subseteq \powset{V}$ and the set of
profiles $\mathcal{F} \subseteq W^V$ are both countable. Key to this definition
and to the results in the paper are conditions on $\mathcal{A}$ and
$\mathcal{F}$ called \emph{uniform measurability} and \emph{quasi-partition
embedding} that ensure their richness and relative compatibility, and which are
substantially weaker than previously proposed alternatives to Arrow's
unrestricted domain condition.
Using this framework we prove the following results.

\begin{thm}
    Arrow's impossibility theorem is provable in $\RCA_0$.
\end{thm}

In \S\ref{sec:ks_arrow} we establish that the Kirman--Sondermann analysis of
social welfare functions for countable (and hence finite) societies in terms of
ultrafilters of decisive coalitions can be formalised in $\RCA_0$. It follows
that Arrow's impossibility theorem is also provable in $\RCA_0$.
Moreover, by replacing finite sets with their codes, Arrow's theorem can be
formalised as a $\Pi^0_1$ sentence which is provable in $\PRA$.

\begin{thm}
    Fishburn's possibility theorem for countable societies is equivalent
    over $\RCA_0$ to the axiom scheme of arithmetical comprehension.
\end{thm}

This shows that Fishburn's possibility theorem requires the same set existence
principles for its proof as theorems of classical analysis like the
Bolzano--Weierstrass theorem, and combinatorial principles like König's
infinity lemma or Ramsey's theorem $\mathrm{RT}^n_k$ for $n \geq 2$ and
$k \geq 3$.
\S\ref{sec:fpt} is devoted to proving this equivalence, which can be seen as an
analogue in second-order arithmetic of theorem~\ref{thm:FPT_equiv_NPU} above.
This result can also be understood as generalising the results of Lewis and
Mihara discussed above to the broader class of countable societies introduced in
\S\ref{sec:societies}.


\section{Preliminaries}
\label{sec:prelims}

This section provides a brief overview of
subsystems of second-order arithmetic (\S\ref{sec:subsystems}),
ultrafilters on countable algebras of sets (\S\ref{sec:algebras}),
and weak orders in social choice theory (\S\ref{sec:orders}).


\subsection{Subsystems of second-order arithmetic}
\label{sec:subsystems}

Reverse mathematics is a subfield of mathematical logic devoted to determining
the set existence principles necessary to prove theorems of ordinary
mathematics, including real and complex analysis, countable algebra, and
countable infinitary combinatorics. This is done by formalising the theorems
concerned in the language of second-order arithmetic, and proving equivalences
between those formalisations and systems located in a well-understood hierarchy
of set existence principles. The equivalence proofs are carried out in a weak
base theory known as $\RCA_0$, which roughly corresponds to computable
mathematics and is briefly described below.
For details of the material in this subsection we refer readers to Simpson's
reference work \emph{Subsystems of Second Order Arithmetic}~\cite{Simpson2009},
Dzhafarov and Mummert's textbook
\emph{Reverse Mathematics}~\cite{DzhafarovMummert2022}, and
Hirschfeldt's monograph \emph{Slicing the Truth}~\cite{Hirschfeldt2014}.

\emph{Second-order arithmetic} $\lang_2$ is a two-sorted formal language, with
\emph{number variables} $x_1,x_2,\dotsc$ whose intended range is the natural
numbers $\N$, and \emph{set variables} $X_1,X_2,\dotsc$ whose intended range is
the powerset of the natural numbers $\powset{\N}$. The non-logical symbols are
those of Peano arithmetic ($0,1,+,\times,<$) plus the $\in$ symbol for set
membership. The atomic formulas of $\lang_2$ are those of the form $t_1 = t_2$,
$t_1 < t_2$, and $t_1 \in X_1$, where $t_1,t_2$ are number terms and $X_1$ is a
set variable. As well as the usual logical connectives, it contains both
\emph{number quantifiers} (sometimes called first-order quantifiers)
$\forall{x}$ and $\exists{x}$, and \emph{set quantifiers} (sometimes called
second-order quantifiers) $\forall{X}$ and $\exists{X}$. Formulas of $\lang_2$
are built up from atomic formulas using logical connectives and set and number
quantifiers.

The base theory $\RCA_0$ has three sets of axioms: the basic arithmetical
axioms, the $\Sigma^0_1$ induction scheme, and the recursive comprehension axiom
scheme. The \emph{basic arithmetical axioms} are those of Peano arithmetic,
minus the induction scheme: in other words, the axioms of a commutative discrete
ordered semiring.
The $\Sigma^0_1$ \emph{induction axiom scheme} consists of the universal
closures of all formulas of the form
\begin{equation}
    \label{eqn:Ind}
    \tag{$\Sigma^0_1\text{-}\mathsf{Ind}$}
    (\varphi(0) \wedge \forall{n}(\varphi(n) \rightarrow \varphi(n+1)))
        \rightarrow \forall{n}\varphi(n),
\end{equation}
where $\varphi$ is a $\Sigma^0_1$ formula, i.e.\ one of the form
$\exists{k}\theta(n,k)$ where $\theta$ contains only bounded quantifiers.
Finally, the \emph{recursive} or $\Delta^0_1$ \emph{comprehension axiom scheme}
consists of the universal closures of all formulas of the form
\begin{equation}
    \label{eqn:RCA}
    \tag{$\Delta^0_1\text{-}\mathsf{CA}$}
    (\varphi(n) \leftrightarrow \psi(n))
        \rightarrow \exists{X}\forall{n}(n \in X \leftrightarrow \varphi(n)),
\end{equation}
where $\varphi$ is a $\Sigma^0_1$ formula and $\psi$ is a $\Pi^0_1$ formula,
i.e.\ one of the form $\forall{k}\theta(n,k)$ where $\theta$ contains only
bounded quantifiers. 

Other subsystems of second-order arithmetic are obtained by extending $\RCA_0$
with additional axioms. The present paper is concerned only with one of these
systems, $\ACA_0$, which is obtained by augmenting the axioms of $\RCA_0$
with the \emph{arithmetical comprehension scheme}, which consists of the
universal closures of all formulas of the form
\begin{equation}
    \label{eqn:ACA}
    \tag{$\mathrm{ACA}$}
    \exists{X}\forall{n}(n \in X \leftrightarrow \varphi(n)),
\end{equation}
where $\varphi$ is an arithmetical formula, i.e.\ which may contain number
quantifiers but no set quantifiers, although it may contain free set variables.


\subsection{Countable algebras and ultrafilters}
\label{sec:algebras}

Our approach to ultrafilters on countable algebras of sets is based on that of
Hirst~\cite{Hirst2004}. We use the standard coding of a sequence of sets by a
single sets using the primitive recursive pairing map $(m,n) = (m + n)^2 + m$.
$Y \subseteq \N$ is a \emph{sequence of sets}, $Y = \str{Y_i : i \in \N}$, if
\begin{equation*}
    (i,v) \in Y \leftrightarrow v \in Y_i
\end{equation*}
for all $i,v \in \N$.

\begin{dfn}[countable algebras of sets]
    \label{dfn:ctble_algebras}
    Let $V \subseteq \N$ and let $\mathcal{A} = \str{A_n : n \in \N}$ be a
    countable sequence of sets such that for every $i \in \N$,
    $A_i \subseteq V$.
    $\mathcal{A}$ is a \emph{countable algebra} over $V$ if it contains $V$ and
    it is closed under unions, intersections, and complements relative to $V$.
    A countable algebra $\mathcal{A}$ over $V$ is \emph{atomic} if for all
    $v \in V$, there exists $k \in \N$ such that $A_k = \{ v \}$.
\end{dfn}

If $\mathcal{A}$ is a countable algebra over a set $V$, we write $A_i^c$ to
denote its relative complement $V \setminus A_i$.
Repetitions are allowed, so given a countable algebra $\mathcal{A}$ we can
computably construct an algebra $\mathcal{A}'$ which contains the same sets
(typically in a different order) in which we can uniformly compute the
operations of complementation, union, and intersection. We make this precise
through the following definition.

\begin{dfn}[boolean embeddings]
    \label{dfn:bool_embed}
    A \emph{boolean formation sequence} is a finite sequence $s \in \Seq$ with
    $\lh{s} \geq 1$ such that for all $j < \lh{s}$, one of the following obtains
    for some $n,m < j$:
    \begin{enumerate}
        \item $s(j) = (0,n,n)$,
        \item $s(j) = (1,n,n)$ and $n < j$,
        \item $s(j) = (2,n,m)$ and $n,m < j$.
    \end{enumerate}
    If $s$ is a boolean formation sequence then we write $s \in \mathrm{BFS}$.
    
    Fix a set $V \subseteq \N$ and suppose that $S = \str{ S_i : i \in \N }$ is
    a countable sequence of subsets of $V$ and that
    $\mathcal{A} = \str{ A_i : i \in \N }$ is an algebra of sets over $V$.
    A function $e : \mathrm{BFS} \to \N$ is a \emph{boolean embedding} of $S$
    into $\mathcal{A}$ if for all boolean formation sequences $s$ with
    $k = \lh{s} - 1$, there exist $n,m < s$ such that
    \begin{enumerate}
        \item If $s(k) = (0,n,n)$ then $A_{e(s)} = S_n$,
        \item If $s(k) = (1,n,n)$ and $n < k$
            then $A_{e(s)} = A_{e(s\restr{n+1})}^c$,
        \item If $s(k) = (2,n,m)$ and $n,m < k$
            then $A_{e(s)} = A_{e(s\restr{n+1})} \cap A_{e(s\restr{m+1})}$.
    \end{enumerate}
\end{dfn}

The following lemma is a straightforward exercise in primitive recursion.

\begin{lem}
    \label{lem:boolean_embedding}
    The following is provable in $\RCA_0$.
    Suppose $S = \str{ S_i : i \in \N }$ is a sequence of subsets of
    $V \subseteq \N$. Then there exists an algebra $\mathcal{A}$ over $V$, a
    boolean embedding $e$ of $S$ into $\mathcal{A}$, and a boolean embedding
    $e^*$ from $\mathcal{A}$ into $\mathcal{A}$.
    
    Moreover, if $S$ is already a countable algebra over $V$, then
    \begin{enumerate}
        \item For all $m \in \N$, $S_m = A_{e(\str{(0,m,m)})}$, and
        \item For all $n \in \N$ there exists $k \in \N$ such that $A_n = S_k$.
    \end{enumerate}
\end{lem}

\begin{dfn}[ultrafilters]
    \label{dfn:ultrafilters}
    Suppose $\mathcal{A} = \str{A_n : n \in N}$ is a countable algebra over
    $V \subseteq \N$. $\Uf \subseteq \N$ is an \emph{ultrafilter} on
    $\mathcal{A}$ if it obeys the following conditions for all $i,j,k \in \N$.
    \begin{enumerate}
        \item (Non-emptiness.)
            If $A_i = V$,
            then $i \in \Uf$.
        \item (Properness.)
            If $A_i = \emptyset$,
            then $i \not\in \Uf$.
        \item (Upwards closure.)
            If $i \in \Uf$ and $A_i \subseteq A_j$,
            then $j \in \Uf$.
        \item (Intersections.)
            If $i,j \in \Uf$ and $A_k = A_i \cap A_j$,
            then $k \in \Uf$.
        \item (Maximality.)
            If $A_j = A_i^c$,
            then $i \in \Uf$ or $j \in \Uf$.
    \end{enumerate}
    An ultrafilter $\Uf$ is \emph{principal} if it obeys the
    following condition, and \emph{non-principal} otherwise.
    \begin{enumerate}[resume]
        \item (Principality.)
            There exist $k,d \in \N$ such that $k \in \Uf$ and $A_k = \{ d \}$.
    \end{enumerate}
\end{dfn}

The next lemma is elementary, but worth stating as it is used a number of times.

\begin{lem}
    \label{lem:uf_basic}
    The following is provable in $\RCA_0$.
    Suppose $\mathcal{A}$ is a countable atomic algebra over $V \subseteq \N$
    and $\Uf \subseteq \N$ is an ultrafilter on $\mathcal{A}$.
    Then $\Uf$ has the following properties for all $i,j,k \in \N$.
    \begin{enumerate}
        \item If $i \in \Uf$ and $A_j = A_i^c$,
            then $j \not\in \Uf$.
        \item If $A_k = A_i \cup A_j$ and $k \in \Uf$,
            then either $i \in \Uf$ or $j \in \Uf$.
        \item Suppose $\str{Y_i : i < k}$ is a finite sequence of sets and
            $s \in \Seq$ is such that $\lh{s} = k + 1$.
            If $Y_i = A_{s(i)}$ for all $i < k$,
            $(\bigcup_{i<k} Y_i) = A_{s(k)}$,
            and $s(k) \in \Uf$,
            then there exists $j < k$ such that $s(j) \in \Uf$.
        \item The following conditions are equivalent:
            \begin{enumerate}
                \item $\Uf$ is principal;
                \item There exists $k \in \N$ such that $A_k$ is finite and
                    $k \in \Uf$;
                \item There exists $d \in V$ such that for all $i \in \N$,
                    $i \in \Uf$ if and only if $d \in A_i$.
            \end{enumerate}
    \end{enumerate}
\end{lem}

When we come to consider Fishburn's possibility theorem in \S\ref{sec:fpt}, we
will need the following well-known result: the existence of non-principal
ultrafilters on countable algebras is equivalent to arithmetical comprehension.
This equivalence appears in its present guise as theorem~9 of
Kreuzer~\cite{Kreuzer2015}, but it has many antecedents. The proof of the
forward direction presented here follows a partition construction from
Kreuzer~\cite{Kreuzer2012}, although similar ideas have been used by others,
going back to Kirby and Paris~\cite{KirbyParis1977} and
Solovay~\cite{Solovay1978}. The reversal uses the fact that non-principal
ultrafilters refine the Fréchet filter in order to code the jump, an idea drawn
from Kirby~\cite[theorem~1.10]{Kirby1984}.

\begin{lem}
    \label{lem:uf_exist_equiv_ACA0}
    The following are equivalent over $\RCA_0$.
    \begin{enumerate}
        \item $\ACA_0$.
        \item For every infinite set $V \subseteq \N$ and every atomic
            countable algebra $\mathcal{A}$ over $V$, there exists a
            non-principal ultrafilter $\Uf$ on $\mathcal{A}$.
    \end{enumerate}
\end{lem}

\begin{proof}
    We first show that 1 implies 2. Working in $\ACA_0$, let $V \subseteq \N$ be
    infinite and let $\mathcal{A}$ be a countable algebra over $V$; we do not
    need the additional assumption that $\mathcal{A}$ is atomic.
    Given $s \in 2^{<\N}$, let
    \begin{equation}
        A^s = \bigcap_{i < \lh{s}} \left\{\begin{array}{ll}
            A_i     & \text{if } s(i) = 0, \\
            (A_i)^c & \text{if } s(i) = 1.
        \end{array}\right.
    \end{equation}
    By $\Sigma^0_0$ induction we have that for all $v \in V$,
    $\forall{n}\exists{!s \in 2^n}(v \in A^s)$. In other words,
    $\str{A^s : s \in 2^n}$ is a partition of $V$.
    To see this, let $t \in 2^n$ be the unique sequence such that $z \in A^t$.
    $v \in A^{t\concat\str{0}} \leftrightarrow v \in A_{n+1}$, so if
    $v \in A_{n+1}$ we set $s = t\concat\str{0}$ and if $v \not\in A_{n+1}$ then
    we set $s = t\concat\str{1}$. Since these possibilities are exclusive,
    either way $s \in 2^{n+1}$ is the unique sequence such that $v \in A^s$ as
    desired.
    Now let
    \begin{equation}
        T = \{ s \in 2^{<\N} : \text{$A^s$ is infinite} \}.
    \end{equation}
    $T$ exists by arithmetical comprehension. We claim that $T$ is an infinite
    tree. Suppose not, so there is some $n$ such that for all $s \in 2^n$,
    $A^s$ is finite. Let $A' = \cup_{s \in 2^n} A^s$. $A'$ is finite since every
    $A^s$ is, so let $m$ bound the elements of $A'$. By assumption $V$ is
    infinite, so there exists $v \in V$ such that $v > m$. $v \not\in A'$ so
    $v \not\in A^s$ for all $s \in 2^n$, contradicting the fact that
    $\str{A^s : s \in 2^n}$ partitions $V$.
    By weak K\"{o}nig's lemma that there exists an infinite path $P$ in $T$, so
    let $\Uf = \{ k : P(k) = 0 \}$, which exists by recursive comprehension in
    the parameter $P$. To complete the proof we show that $\Uf$ is a
    non-principal ultrafilter on $\mathcal{A}$.
    
    To establish non-principality it suffices to note that every
    $A_i$ such that $i \in \Uf$ is infinite because $A^{P \restr {i + 1}}$ is an
    infinite subset of $A_i$.
    To show maximality, let $i \in \Uf$ be arbitrary with $(A_i)^c = A_j$,
    and suppose $j \in \Uf$. Let $k = \max \{ i,j \} + 1$. Since, by our
    assumption, $P(i) = P(j) = 0$, we have that $A^{P \restr k} = \emptyset$,
    contradicting the fact that $P \restr k \in T$
    and so $A^{P \restr k}$ is infinite.
    To show that $\Uf$ is closed under intersections, let
    $i,j \in \Uf$, let $A_m = A_i \cap A_j$, and let
    $A_n = (A_m)^c$. Suppose for a contradiction that $m \not\in \Uf$,
    so by maximality $n \in \Uf$. Let $k = \max \{ i,j,m,n \} + 1$.
    Then $A^{P \restr k} = \emptyset$, contradicting the fact that
    $P \restr k \in T$.
    A similar argument establishes upwards closure. Take $i \in \Uf$ and
    suppose $A_i \subseteq A_j$. Towards a contradiction assume that
    $j \not \in \Uf$, so by maximality and intersections if
    $A_k = A_i \cap (A_j)^c = \emptyset$ then $k \in \Uf$, contradicting
    non-principality.
    
    
    Working now in $\RCA_0$, we show that 2 implies 1.
    To prove arithmetical comprehension it suffices to prove that the range of
    any one-to-one function $h : \N \to \N$ exists
    \cite[lemma III.1.3, pp.~105--106]{Simpson2009}.
    The sequence
    \begin{equation}
        B = \{ (2n,  v) : v \in V \wedge (\exists{k < v})(h(k) = n) \} \cup
            \{ (2n+1,n) : n \in V \}
    \end{equation}
    exists by recursive comprehension since all quantifiers in its definition
    are bound\-ed, and by lemma~\ref{lem:boolean_embedding} there exists a
    countable algebra $\mathcal{A} = \str{A_i : i \in \N}$ over $\N$ and a
    boolean embedding $e : \mathrm{BFS} \to \N$ of $B$ into $\mathcal{A}$.
    The right-hand-side of the union defining $B$ ensures that $\mathcal{A}$ is
    atomic, i.e.\ it contains all singletons $\{v\}$ for $v \in V$.
    
    For convenience we write $n'$ to mean $e(\str{(0,2n,2n)})$, i.e.\ the index
    in $\mathcal{A}$ such that $A_{n'} = B_n$.
    
    By 2 there exists $\Uf \subseteq \N$ such that $\Uf$ is a non-principal
    ultrafilter on $\mathcal{A}$, and by recursive comprehension, the set
    $Y = \{ n : n' \in \Uf \}$ exists.
    We show that $Y = \ran(h) = \{ n : \exists{k}(h(k) = n) \}$.
    
    Suppose $n \in Y$, so $n' \in \Uf$ and thus by non-principality
    $A_{n'}$ is non-empty, meaning there is some $v$ such that
    $(\exists{k<v})(h(k) = n)$. It follows that $\exists{k}(h(k) = n)$,
    i.e.\ $n \in \ran(h)$.
    For the converse note that if $\exists{k}(h(k) = n)$ then
    $A_{n'}$ is cofinite. To see this, fix any $m \in A_{n'}$ and any
    $j \in \N$. Assume $v + j \in A_{n'}$, so there exists some
    $k < v + j$ such that $h(k) = n$. $k < v + j + 1$, so by $\Sigma^0_0$
    induction, for all $j$, $v + j \in A_{n'}$. Consequently $A_{n'}$ is
    cofinite and so by maximality $n' \in \Uf$, and thus $n \in Y$.
\end{proof}


\subsection{Orderings in social choice theory}
\label{sec:orders}

The paper aims to be self-contained where notions from social choice theory are
concerned, but a good starting point for a deeper study is Taylor's monograph
\emph{Social Choice and the Mathematics of Manipulation}~\cite{Taylor2005}.
In social choice theory, voters express their preferences as orders on the set
of alternatives $X$ (e.g.\ ranking candidates in an election). These orders are
required to be transitive and strongly connected, but ties are permitted to
express indifference between alternatives. This notion is standardly called a
\emph{weak order} in the social choice theory literature, and we follow this
terminology here, noting that it is synonymous with the notion of a total
preorder. In this paper we will be concerned exclusively with finite sets of
alternatives $X$, and hence all our weak orders will be assumed to be coded by
natural numbers.

\begin{dfn}[weak orders]
    \label{dfn:orders}
    Suppose $X \subseteq \N$ is nonempty and $R \subseteq X \times X$.
    $R$ is \emph{strongly connected} if $(x,y) \in R$ or $(y,x) \in R$ for all
    $x,y \in X$.
    
    If $R$ is a transitive and strongly connected relation then we call it a
    \emph{weak order} and write $x \lesssim_R y$ to mean $(x,y) \in R$,
    $x <_R y$ to mean $(x,y) \in R \wedge (y,x) \not\in R$, and $x \sim_R y$ to
    mean $(x,y) \in R \wedge (y,x) \in R$.
\end{dfn}

Many basic properties of weak orders can be established in $\RCA_0$.
For example, if $R$ is a weak order then
\begin{enumerate}
    \item $\lesssim_R$ is reflexive;
    \item $\sim_R$ is an equivalence relation on $X$;
    \item If $x <_R z$ then $x <_R y$ or $y <_R z$ (negative transitivity).
\end{enumerate}

Given a set $V \subseteq \N$ of voters and a finite set $X \subseteq \N$ of
alternatives, we let $W$ be the set of all (codes for) weak orders on $X$. A
\emph{profile} is a function $f : V \to W$. In practice we will always be
concerned with countable sequences $\mathcal{F} = \str{ f_i : i \in \N }$ of
profiles. If $f_i$ is a profile and $v \in V$ is a voter then we write
$x \lesssim_{i(v)} y$ to mean that $x \lesssim_R y$ where $R = f_i(v)$, i.e.\
that alternative $x$ is preferred to $y$ by voter $v$ in the voting scenario
represented by the profile $f_i$. Similarly we write $x <_{i(v)} y$ to mean
$x <_R y$, and $x \sim_{i(v)} y$ to mean $x \sim_R y$.

A \emph{coalition} is simply a set $C \subseteq V$ of voters; by convention, we
allow both the empty set and singleton sets containing only one voter to count
as coalitions. Given a coalition $C$, we write $x \lesssim_{i[C]} y$ to mean
that $x \lesssim_{i(v)} y$ for all $v \in C$, and $x <_{i[C]} y$ and
$x \sim_{i[C]} y$ have their obvious meanings.

If $Y \subseteq X$, we write \emph{$f_i(v) = f_j(v)$ on $Y$} to mean that
$x \lesssim_{i(v)} y \leftrightarrow x \lesssim_{j(v)} y$ for all $x,y \in Y$,
i.e.\ that $v$'s preferences regarding all $x$ and $y$ in $S$ are the same under
both the voting scenarios represented by the profiles $f_i$ and $f_j$.
We write \emph{$f_i = f_j$ on $Y$} to mean that $f_i(v) = f_j(v)$ on
$Y$ for all $v \in V$.


\section{Countable societies}
\label{sec:societies}

In the classical social choice literature, the notion of a society has been
generalised by Armstrong~\cite{Armstrong1980} to allow $\mathcal{A}$ to be any
algebra of sets over $V$, rather than all of $\powset{V}$. In Armstrong's
generalisation $\mathcal{F}$ is always the set of all $\mathcal{A}$-measurable
profiles, i.e.\ those $f : V \to W$ such that for all $x,y \in X$,
$\{ v : x \lesssim_{f(v)} y \} \in \mathcal{A}$.
This paper only addresses the countable case, i.e.\ when not only $V$ but also
$\mathcal{A}$ and $\mathcal{F}$ are countable objects that can be coded by
sets of natural numbers.%
\footnote{%
A different approach, following that of Towsner~\cite{Towsner2014}, would be to
introduce new symbols $\mathfrak{U}$ and $\mathfrak{S}$ standing for third-order
objects like ultrafilters and social welfare functions. However, the approach
via countable algebras pursued in this paper is more congenial to both the
reverse mathematics and the underlying motivation of viewing social welfare
functions as potentially computable (and hence countable) objects.
}

A \emph{countable society} consists of a set of voters $V \subseteq \N$, a
finite set of alternatives $X \subseteq \N$ and the associated set $W$ of weak
orders on $X$, an atomic countable algebra of coalitions $\mathcal{A}$, and a
countable sequence of profiles $\mathcal{F} = \str{ f_i : i \in \N }$ over
$V,X$ (i.e.\ for all $i$, $f_i$ is a function from $V$ to $W$). However, in
order for theorems about countable societies to continue to make sense in the
way they do when $\mathcal{A} = \powset{V}$ and $\mathcal{F} = W^V$, we need to
impose certain conditions on $\mathcal{A}$ and $\mathcal{F}$. The first such
condition is that profiles in $\mathcal{F}$ are measurable by coalitions in
$\mathcal{A}$. Measurability must also be uniform, to ensure that proofs using
it can be carried out in $\RCA_0$.

\begin{dfn}[uniform measurability]
    \label{dfn:uniform_measurability}
    Suppose $V \subseteq \N$ is nonempty and $X \subseteq \N$ is finite, and
    that $\mathcal{A}$ is a countable algebra of sets over $V$ and $\mathcal{F}$
    is a countable sequence of profiles over $V,X$.
    If there exists $\mu : \N \times X \times X \to \N$ such that for all
    $n \in \N$, $x,y \in X$, and $v \in V$,
    \begin{equation*}
        x \lesssim_{n(v)} y \leftrightarrow v \in A_{\mu(n,x,y)},
    \end{equation*}
    then we say $\mathcal{F}$ is \emph{uniformly $\mathcal{A}$-measurable}.
\end{dfn}

\begin{lem}
    \label{lem:unif_embed}
    The following is provable in $\RCA_0$.
    Suppose $V \subseteq \N$ is nonempty and $X \subseteq \N$ is nonempty and
    finite, and that $\mathcal{A} = \str{ A_i : i \in \N }$ is a countable
    algebra of sets over $V$ and $\mathcal{F} = \str{ f_i : i \in \N }$ is a
    countable sequence of profiles over $V,X$.
    If $\mathcal{F}$ is uniformly $\mathcal{A}$-measurable then there exist
    functions $\mu_{<}, \mu_{\sim} : \N \times X \times X \to \N$ such that
    for all $n \in \N$, $x,y \in X$, and $v \in V$,
    \begin{equation*}
        x <_{n(v)} y \leftrightarrow v \in A_{\mu_{<}(n,x,y)}
    \end{equation*}
    and
    \begin{equation*}
        x \sim_{n(v)} y \leftrightarrow v \in A_{\mu_{\sim}(n,x,y)}.
    \end{equation*}
\end{lem}

The second condition, \emph{quasi-partition embedding}, ensures that finite
sequences of coalitions in $\mathcal{A}$ can be recovered uniformly from
profiles in $\mathcal{F}$. This condition emerges naturally from the proofs of
the Kirman--Sondermann theorem and Fishburn's possibility theorem, although to
the best of our knowledge it is isolated here for the first time.%
\footnote{%
Other weakenings of Arrow's universal domain condition are well-known, such as
the \emph{free triple property} and the \emph{chain property}, but when $V$ is
infinite these conditions still guarantee that $\mathcal{F}$ is uncountable.
}
Quasi-partitions of $V$, in which overlaps are allowed, are preferred to
partitions since they are more computationally tractable.%
\footnote{%
One way of thinking of this condition is as providing a uniform way of
transforming finite covers of $V$ into finite partitions of $V$.
}

\begin{dfn}[quasi-partition embedding]
    \label{dfn:qp_embedding}
    Suppose $V \subseteq \N$ is nonempty and $X \subseteq \N$ is finite with
    $|X| \geq 3$, and that $\mathcal{A}$ is a countable algebra of sets over $V$
    and $\mathcal{F}$ is a countable profile algebra over $V,X$.
    A \emph{permutation} of a finite set $W$ is a finite sequence $p \in \Seq$
    such that for all (codes for) weak orders $R \in W$ there exists a unique
    $i$ such that $p(i) = R$. We write $p \in \mathrm{Perm}(W)$ to indicate that
    $p$ is a permutation of $W$.
    A \emph{quasi-partition} is a finite sequence $s \in \Seq$ such that
    $1 \leq \lh{s}$. We write $s \in \mathrm{QPart}(k)$ to indicate that $s$ is
    a quasi-partition with $\lh{s} \leq k$.
    $\mathcal{A}$ is \emph{quasi-partition embedded} into $\mathcal{F}$ if there
    exists a function
    $e : \mathrm{Perm}(W) \times \mathrm{QPart}(\lh{W}) \to \N$
    such that for all $v \in V$,
    \begin{equation*}
        f_{e(p,s)}(v) = \begin{cases}
            p(i)          & \text{if }(\exists{!i < \lh{s}-1})
                                (v \in A_{s(i)}), \\
            p(\lh{s} - 1) & \text{otherwise}.
        \end{cases}
    \end{equation*}
\end{dfn}

\begin{dfn}[countable societies]
    A \emph{countable society} $\mathcal{S}$ consists of a nonempty set
    $V \subseteq \N$ of voters, a finite set $X \subseteq \N$ of alternatives
    with $|X| \geq 3$, an atomic countable algebra $\mathcal{A}$ over $V$, and a
    sequence $\mathcal{F} = \str{ f_i : i \in \N }$ of profiles over $V,X$ such
    that $\mathcal{F}$ is uniformly $\mathcal{A}$-measurable and $\mathcal{A}$
    is quasi-partition embedded into $\mathcal{F}$.
    
    A countable society $\mathcal{S}$ is \emph{finite} if $V$ is finite, and
    \emph{infinite} otherwise.
\end{dfn}

\begin{dfn}[social welfare functions]
    \label{dfn:social_welfare_fns}
    Suppose that $\mathcal{S}$ is a countable society.
    $\sigma : \N \to W$ is a \emph{social welfare function} for $\mathcal{S}$
    if it obeys the following conditions.
    \begin{enumerate}
        \item (Unanimity.)
            For all $x, y \in X$ and $i \in \N$,
            if $x <_{i[V]} y$ then $x <_{\sigma(i)} y$.
        \item (Independence.)
            For all $x,y \in X$ and all $i,j \in \N$, if $f_i = f_j$ on
            $\{ x, y \}$ then $\sigma(i) = \sigma(j)$ on $\{ x, y \}$.
    \end{enumerate}
    If $\sigma$ obeys the following additional condition then it is
    \emph{non-dictatorial}.
    \begin{enumerate}[resume]
        \item (Non-dictatoriality.)
            For all $v \in V$ there exists $i \in \N$ and $x,y \in X$ such that
            $x <_{i(v)} y$ and $y \lesssim_{\sigma(i)} x$.
    \end{enumerate}
\end{dfn}

\begin{dfn}[decisive coalitions]
    \label{dfn:decisive}
    Suppose $\mathcal{S} = \str{V,X,\mathcal{A},\mathcal{F}}$ is a countable
    society and that $\sigma$ is a social welfare function for $\mathcal{S}$.
    \begin{enumerate}
        \item $A_n$ is \emph{$\sigma$-decisive for $x,y$} if for all $i$,
            $x <_{i[A_n]} y$ implies $x <_{\sigma(i)} y$.
        \item $A_n$ is \emph{$\sigma$-decisive} if it is $\sigma$-decisive for
            all $x,y \in X$.
        \item $A_n$ is \emph{almost $\sigma$-decisive for $x,y$ at $i$} if
            $x <_{i[A_n]} y$, $y <_{i[A_n^c]} x$, and $x <_{\sigma(i)} y$.
        \item $A_n$ is \emph{almost $\sigma$-decisive for $x,y$} if
            \begin{equation*}
            \forall{i}(
                (x <_{i[A_n]} y \wedge y <_{i[A_n^c]} x)
                \rightarrow x <_{\sigma(i)} y
            ).
            \end{equation*}
        \item $A_n$ is \emph{almost $\sigma$-decisive} if it is
            almost $\sigma$-decisive for all $x,y \in X$.
    \end{enumerate}
\end{dfn}

The notion of a decisive coalition is due to
Arrow~\cite[definition~10, p.~52]{Arrow1963}, while almost decisiveness was
introduced by Sen~\cite[definition~3*2, p.~42]{Sen1970a}.
The non-dictatoriality condition for social welfare functions can be rephrased
in terms of decisive coalitions, namely by saying that no singleton
$\{ d \} \subseteq V$ is $\sigma$-decisive. This gives rise to some natural
strengthenings of non-dictatoriality (definition~\ref{dfn:fpt}).


\section{Arrow's theorem via ultrafilters}
\label{sec:ks_arrow}

In this section we show how the Kirman--Sondermann analysis of social welfare
functions in terms of ultrafilters can be carried out in $\RCA_0$
(theorem~\ref{thm:KS_RCA0}). This immediately gives a proof of Arrow's theorem
in $\RCA_0$ (theorem~\ref{thm:arrow_RCA0}).

\begin{dfn}
    \label{dfn:KS_Arrow}
    The \emph{Kirman--Sondermann theorem for countable societies}
    ($\KS$) is the following statement:
    Suppose $\mathcal{S} = \str{V,X,\mathcal{A},\mathcal{F}}$ is a countable
    society and that $\sigma$ is a social welfare function for $\mathcal{S}$.
    Then there exists an ultrafilter
    \begin{equation*}
        \Uf_\sigma = \{ i : \text{$A_i$ is $\sigma$-decisive} \}
    \end{equation*}
    on $\mathcal{A}$ which is principal if and only if $\sigma$ is dictatorial.
    
    \emph{Arrow's theorem} is the statement that if $\mathcal{S}$ is a
    finite society and $\sigma$ is a social welfare function for $\mathcal{S}$,
    then $\sigma$ is dictatorial.
\end{dfn}

A crucial step in many proofs of Arrow's theorem is sometimes known in the
social choice literature as the
``spread of decisiveness'' \cite[pp.~35--37]{Sen2014} or the
``contagion lemma'' \cite[pp.~44--45]{CampbellKelly2002}.
Kirman and Sondermann's version of this is a lemma showing that there exists a
profile $f$ and a pair of alternatives $x,y \in X$ such that $C \subseteq V$ is
almost $\sigma$-decisive at $f$ for $x,y$ if and only if $C$ is almost
$\sigma$-decisive for every profile and every pair of alternatives
\cite[lemma~A]{KirmanSondermann1972}. In our arithmetical setting, the
corresponding versions of these two conditions are $\Sigma^0_2$ and $\Pi^0_2$
respectively, so formalising Kirman and Sondermann's lemma~A establishes that
the set $\{ i : \text{$A_i$ is almost $\sigma$-decisive} \}$ is $\Delta^0_2$
definable relative to $\mathcal{S}$ and $\sigma$. However, the definition of a
countable society in fact allows us to uniformly find witnesses for this last
condition, and thereby obtain a $\Sigma^0_0$ definition.

This and subsequent proofs are made easier by the use of some notation for weak
orders. Given distinct alternatives $x,y,z \in X$,
\begin{equation*}
    R = x < y < z \sim \ast
\end{equation*}
means that $R$ is a weak order such that $x <_R y$ and $y <_R z$, and hence
$x <_R z$. We use the wildcard symbol $\ast$ to quantify over all $c \in X$ not
explicitly mentioned, so in the example above, any other $c \in X$ is such that
$y <_R c$ but $z \sim_R c$.
This notation thus denotes a unique weak order, or rather, the natural number
coding it as a finite set.

\begin{lem}
    \label{lem:uf_defin}
    The following is provable in $\RCA_0$.
    Suppose $\mathcal{S} = \str{V,X,\mathcal{A},\mathcal{F}}$ is a countable
    society and $\sigma$ is a social welfare function for $\mathcal{S}$.
    Then there exists a function $g : \N \to \N$ and alternatives
    $a,b \in X$ such that the following conditions are equivalent
    for all $n \in \N$.
    \begin{enumerate}
        \item $A_n$ is almost $\sigma$-decisive.
        \item There exist $x,y \in X$ such that $A_n$ is almost
            $\sigma$-decisive for $x,y$.
        \item There exist $k \in \mathbb{N}$ and $x,y \in X$ such that
            $A_n$ is almost $\sigma$-decisive for $x,y$ at $k$.
        \item $a <_{\sigma(g(i))} b$.
    \end{enumerate}
\end{lem}

\begin{proof}
    It follows immediately from the statements that 1 implies 2,
    and 2 implies 3.
    We show that 3 implies 2. Let $f_m$ be arbitrary, let $A_n$ be almost
    $\sigma$-decisive for $x,y$ at $k$, and assume that $x <_{m[A_n]} y$ and
    $y <_{m[A_n^c]} x$.
    Given $v \in V$, if $v \in A_n$ then $x <_{k(v)} y$ by almost
    $\sigma$-decisiveness and $x <_{m(v)} y$ by assumption, while if
    $v \not\in A_n$ then $y <_{k(v)} x$ and $y <_{m(v)} x$, so $f_m = f_k$ on
    $\{ x, y \}$ and thus $\sigma(m) = \sigma(k)$ on $\{ x, y \}$ by
    independence. Since $x <_{\sigma(k)} y$ it follows that $x <_{\sigma(m)} y$,
    establishing that $A_n$ is almost $\sigma$-decisive for $x,y$.
    
    
    Now we show that 2 implies 1. Let $A_n$ be almost $\sigma$-decisive
    for $x,y$ and let $z \in X \setminus \{ x, y \}$. Assume that
    $x <_{m[A_n]} z$ and $z <_{m[A_n^c]} x$ for some $f_m$.
    Since $\mathcal{F}$ quasi-partition embeds $\mathcal{A}$, there exists
    $j$ such that
    \begin{equation*}
        f_j(v) = \left\{\begin{array}{ll}
            x < y < z \sim \ast & \text{if } v \in A_n, \\
            y < z < x \sim \ast & \text{if } v \in A_n^c.
        \end{array}\right.
    \end{equation*}
    By the almost $\sigma$-decisiveness of $A_n$ and the construction of $f_j$,
    it follows that $x <_{\sigma(j)} y$, and by unanimity, $y <_{\sigma(j)} z$,
    so by transitivity we have that $x <_{\sigma(j)} z$.
    By our initial assumption, and the construction of $f_j$, $f_m = f_j$ on
    $\{ x, z \}$, so by independence $x <_{\sigma(m)} z$.
    
    A similar argument yields that
    \begin{equation*}
        \left(z \lesssim_{m[A_n]} y \wedge y \lesssim_{m[A_n^c]} z\right)
            \rightarrow z \lesssim_{\sigma(m)} y.
    \end{equation*}
    
    Now fix $w \in X$. If $w \in \{ x, y, z \}$ we are done, so assume
    otherwise. Running the argument twice more we get that
    \begin{equation*}
        \left(z \lesssim_{m[A_n]} w \wedge w \lesssim_{m[A_n^c]} z\right)
            \rightarrow z \lesssim_{\sigma(m)} w,
    \end{equation*}
    and since $w,z$ were arbitrary, we have established that $A_n$ is
    almost $\sigma$-decisive.
    
    
    Finally we show that $g$ exists and that 1 and 4 are equivalent.
    Pick any $a,b \in X$ and let $p$ be a permutation of $W$ such that
    $p(0) = a < b < *$ and $p(1) = b < a < *$.
    $\mathcal{A}$ is quasi-partition embedded into $\mathcal{F}$ by some
    $e : \N \to \N$, so we have
    \begin{equation*}
        f_{e(p,\str{n})}(v) = \begin{cases}
            a < b < * & \text{if } v \in A_n, \\
            b < a < * & \text{if } v \in A_n^c.
        \end{cases}
    \end{equation*}
    The function $g(n) = e(p,\str{n})$ exists by recursive comprehension.
    
    If $A_n$ is almost $\sigma$-decisive then $a <_{\sigma(g(n))} b$ by the
    definition of $g$, so suppose for the converse implication that
    $a <_{\sigma(g(n))} b$.
    $a <_{g(n)[A_n]} b$ and $b <_{g(n)[A_n^c]} a$ by the definition of $g$,
    meaning $A_n$ is almost $\sigma$-decisive for $a,b$ at $g(n)$. By the
    equivalence between 1 and 3, $A_n$ is almost $\sigma$-decisive.
\end{proof}

\begin{lem}
    \label{lem:uf_sigma_ultrafilter}
    The following is provable in $\RCA_0$.
    Suppose $\mathcal{S} = \str{V,X,\mathcal{A},\mathcal{F}}$ is a countable
    society and $\sigma$ is a social welfare function for $\mathcal{S}$.
    Then the set
    \begin{equation*}
        \Uf_\sigma = \{ i \in \N : \text{$A_i$ is almost $\sigma$-decisive} \}
    \end{equation*}
    exists and forms an ultrafilter on $\mathcal{A}$.
\end{lem}

\begin{proof}
    Working in $\RCA_0$, fix a countable society
    $\mathcal{S} = \str{V,X,\mathcal{A},\mathcal{F}}$
    and a social welfare function $\sigma$ for $\mathcal{S}$.
    For all of the arguments below we fix distinct $x,y,z \in X$.
    
    
    To show that $\Uf_\sigma$ exists, note that by lemma~\ref{lem:uf_defin}
    there exists a function $g : \N \to \N$ and $x,y \in X$ such that
    \begin{equation*}
        x <_{\sigma(g(i))} y \leftrightarrow
            \text{$A_i$ is almost $\sigma$-decisive}.
    \end{equation*}
    The left-hand side of this definition is $\Sigma^0_0$ in the parameters
    $\mathcal{S},\sigma,g$, so $\Uf_\sigma$ exists by recursive comprehension in
    those parameters.
    In the remainder of the proof we show that $\Uf_\sigma$ is an ultrafilter on
    $\mathcal{A}$.
    
    
    That $\Uf_\sigma$ contains an index for $V$ and no index for $\emptyset$
    follows straightforwardly from unanimity, so we next prove upwards closure
    under the subset relation. Suppose $i \in \Uf_\sigma$ and
    $A_i \subseteq A_j$, and partition $V$ into
    \begin{align*}
        V_0 &= A_i, \\
        V_1 &= A_i^c \cap A_j, \\
        V_2 &= A_j^c.
    \end{align*}
    Since $\mathcal{A}$ is quasi-partition embedded into $\mathcal{F}$, there
    exists some $m$ such that
    \begin{equation*}
        f_m(v) = \begin{cases}
            x < y < z \sim \ast & \text{if } v \in V_0, \\
            y < x < z \sim \ast & \text{if } v \in V_1, \\
            y < z < x \sim \ast & \text{if } v \in V_2.
        \end{cases}
    \end{equation*}
    $x <_{m[A_i]} y$ by the definition of $f_m$, so since $A_i$ is almost
    $\sigma$-decisive we have that $x <_{\sigma(m)} y$.
    The definition of $f_m$ also gives us that $y <_{m[V]} z$, so by unanimity,
    $y <_{\sigma(m)} z$, and by transitivity, $x <_{\sigma(m)} z$, which
    suffices to establish that $j \in \Uf_\sigma$ by clause~3 of
    lemma~\ref{lem:uf_defin}.
    
    
    Next we prove that $\Uf_\sigma$ is closed under intersections. Suppose that
    $i,j \in \Uf_\sigma$ and that $k$ is such that $A_k = A_i \cap A_j$.
    Partition $V$ into
    \begin{align*}
        V_1 &= A_i   \cap A_j,   \\
        V_2 &= A_i   \cap A_j^c, \\
        V_3 &= A_i^c \cap A_j,   \\
        V_4 &= A_i^c \cap A_j^c.
    \end{align*}
    By quasi-partition embedding let the profile $f_n$ be defined as follows.
    \begin{equation*}
        f_n(v) = \begin{cases}
            z < x < y \sim \ast & \text{if } v \in V_1, \\
            x < y < z \sim \ast & \text{if } v \in V_2, \\
            y < z < x \sim \ast & \text{if } v \in V_3, \\
            y < x < z \sim \ast & \text{if } v \in V_4.
        \end{cases}
    \end{equation*}
    Since $A_i = V_1 \cup V_2$ we have that $x <_{n[A_i]} y$ by the definition of
    $f_n$. Similarly since $A_i^c = V_3 \cup V_4$, $y <_{n[A_i^c]} x$, so by the
    almost $\sigma$-decisiveness of $A_i$ it follows that $x <_{\sigma(n)} y$.
    By a parallel piece of reasoning we have that $z <_{\sigma(n)} x$, and so by
    transitivity $z <_{\sigma(n)} y$. It follows by clause~3 of
    lemma~\ref{lem:uf_defin} that $A_k$ is almost $\sigma$-decisive.
    
    
    Finally we prove that $\Uf_\sigma$ satisfies maximality.
    Suppose that $A_j = A_i^c$. By quasi-partition embedding there exists some
    $m \in \N$ such that
    \begin{equation*}
        f_m(v) = \begin{cases}
            y < z < x \sim \ast & \text{if } v \in A_i, \\
            x < y < z \sim \ast & \text{if } v \in A_i^c.
        \end{cases}
    \end{equation*}
    By unanimity we have that $y <_{\sigma(m)} z$, so either $y <_{\sigma(m)} x$
    or $x <_{\sigma(m)} z$. In the former case, $m,y,x$ witness that $A_i$ is
    almost $\sigma$-decisive by clause~3 of lemma~\ref{lem:uf_defin}, while in
    the latter case $m,x,z$ witness that $A_j$ is almost $\sigma$-decisive.
\end{proof}

\begin{thm}
    \label{thm:KS_RCA0}
    $\KS$ is provable in $\RCA_0$.
\end{thm}

\begin{proof}
    We work in $\RCA_0$. Let $\mathcal{S} = \str{V,X,\mathcal{A},\mathcal{F}}$
    be a countable society and let $\sigma : \N \to X$ be a social welfare
    function for $\mathcal{S}$.
    By lemma~\ref{lem:uf_sigma_ultrafilter} there exists an ultrafilter
    $\Uf_\sigma \subseteq \N$ on $\mathcal{A}$ such that
    \begin{equation*}
        \Uf_\sigma = \{ i \in \N : \text{$A_i$ is almost $\sigma$-decisive} \}.
    \end{equation*}
    It only remains to be shown that (i) $i \in \Uf_\sigma$ if and only if
    $A_i$ is $\sigma$-decisive, and (ii) $\Uf_\sigma$ is principal if and only
    if $\sigma$ is dictatorial.
    
    For (i), the backwards direction is immediate from the definitions.
    For the forward direction fix $A_i$ such that $i \in \Uf_\sigma$,
    i.e.\ $A_i$ is almost $\sigma$-decisive. Let $f_m$ and $x,y$ be such that
    $x <_{m[A_i]} y$; we will establish that $x <_{\sigma(m)} y$.
    
    Start by partitioning $V$ into the sets
    \begin{align*}
        V_0 &= \{ v : x <_{m(v)} y \}, \\
        V_1 &= \{ v : y <_{m(v)} x \}, \\
        V_2 &= \{ v : x \sim_{m(v)} y \} = (V_0 \cup V_1)^c.
    \end{align*}
    By uniform $\mathcal{A}$-measurability, there exist $e_0,e_1,e_2 \in \N$
    such that $A_{e_j} = V_j$ for all $j \leq 2$, and because $\mathcal{F}$
    quasi-partition embeds $\mathcal{A}$, there exists $n \in \N$ such that
    \begin{equation*}
        f_n(v) = \begin{cases}
            x < z < y \sim \ast    & \text{if } v \in V_0, \\
            y < z < x \sim \ast    & \text{if } v \in V_1, \\
            x \sim y < z \sim \ast & \text{if } v \in V_2.
        \end{cases}
    \end{equation*}
    
    By hypothesis we have that $A_i$ is almost $\sigma$-decisive and
    $A_i \subseteq V_0$, so $V_0$ is almost $\sigma$-decisive by upwards
    closure. $V_0^c = V_1 \cup V_2$, so since $z <_{n[V_0]} y$ and
    $y <_{n[V_1 \cup V_2]} z$, it follows from the almost $\sigma$-decisiveness
    of $V_0$ that $z <_{\sigma(n)} y$.
    
    Let $e_3$ be such that $A_{e_3} = V_0 \cup V_2$, and hence
    $A_{e_3}^c = V_1$. By upwards closure again, $A_{e_3}$ is almost
    $\sigma$-decisive, and so because $x <_{n[A_{e_3}]} z$ and
    $z <_{n[A_{e_3}^c]} x$, $x <_{\sigma(n)} z$. It follows by transitivity that
    $x <_{\sigma(n)} y$.
    Finally, by definition $f_n = f_m$ on $\{ x, y \}$, and so by independence
    $x <_{\sigma(m)} y$ as desired.
    
    For the forward direction of (ii), assume that there exist $k,d$ such that
    $A_k = \{ d \}$ and $k \in \Uf_\sigma$. It follows from i) that $A_k$ is
    $\sigma$-decisive, and so $d$ is a dictator for $\sigma$.
    
    
    For the backwards direction of (ii), suppose that $\sigma$ has a dictator
    $d \in V$. $\mathcal{A}$ is atomic, so let $k$ be any index such that
    $A_k = \{ d \}$.
    Since $\mathcal{F}$ quasi-partition embeds $\mathcal{A}$, there exists an
    $n$ such that $f_n$ is defined as follows.
    \begin{equation*}
        f_n(v) = \begin{cases}
            x < y \sim \ast & \text{if } v \in A_k, \\
            y < x \sim \ast & \text{if } v \in A_k^c.
        \end{cases}
    \end{equation*}
    By the definition of $f_n$ we have that $x <_{n[A_k]} y$ and
    $y <_{n[A_k^c]} x$, and by the dictatoriality of $d$ we have that
    $x <_{\sigma(n)} y$, so by lemma~\ref{lem:uf_defin}, $k \in \Uf_\sigma$
    and hence $\Uf_\sigma$ is principal.
\end{proof}

\begin{thm}
    \label{thm:arrow_RCA0}
    Arrow's theorem is provable in $\RCA_0$.
\end{thm}

\begin{proof}
    We work in $\RCA_0$.
    Suppose $\mathcal{S} = \str{V,X,\mathcal{A},\mathcal{F}}$ is a finite
    society, and let $\sigma : \N \to W$ be any social welfare function for
    $\mathcal{S}$.
    By $\KS$ (theorem~\ref{thm:KS_RCA0}), there exists an ultrafilter
    $\Uf_\sigma$ on $\mathcal{A}$ which is principal if and only if $\sigma$ is
    dictatorial. Since $V$ is finite, $\Uf_\sigma$ is principal by part~4 of
    lemma~\ref{lem:uf_basic}. Therefore, $\sigma$ is dictatorial.
\end{proof}

Since all the objects involved in Arrow's theorem are finite, it can be
formalised as a sentence $\theta$ in the language of first-order arithmetic, by
replacing quantification over finite sets of natural numbers with quantification
over the numbers that code them (for details see \S II.2 of \cite{Simpson2009}
or \S 5.5.2 of \cite{DzhafarovMummert2022}). The first-order sentence $\theta$
then follows in $\RCA_0$ from the second-order statement of Arrow's theorem in
virtue of the coding. As long as one is careful with writing down the relevant
bounds, $\theta$ will be a $\Pi^0_1$ statement, i.e.\ of the form
$\forall{n}\psi(n)$ where $\psi(n)$ contains only bounded quantifiers. By
results of Friedman~\cite{Friedman1976} and Parsons~\cite{Parsons1970}, $\RCA_0$
is conservative over primitive recursive arithmetic ($\PRA$) for all $\Pi^0_2$
statements \cite[\S IX.1]{Simpson2009}. We therefore have that Arrow's theorem
(in the form of its first-order formalisation $\theta$) is provable in $\PRA$,
and hence it is \emph{finitarily provable} in the sense of Tait's analysis of
Hilbert's program \cite{Tait1981}. Moreover, the bounds in $\theta$ are
exponential, which suggests the following stronger result.

\begin{conj}
    \label{conj:arrow_id0exp}
    The first-order formalisation of Arrow's theorem is provable in
    $\mathrm{I}\Delta_0 + \mathrm{exp}$.
\end{conj}


\section{Fishburn's possibility theorem}
\label{sec:fpt}

The main result of this section, theorem~\ref{thm:FPT_equiv_ACA0} is that
Fishburn's possibility theorem for countable societies is equivalent to $\ACA_0$
over $\RCA_0$.
We also show that non-dictatorial social welfare functions actually satisfy more
general non-dictatoriality conditions than Arrow's original condition
(lemma~\ref{lem:nd_equiv_cofcoal}).

\begin{dfn}
    \label{dfn:fpt}
    \emph{Fishburn's possibility theorem for countable societies}
    ($\FPT$) is the following statement: For all countable societies
    $\mathcal{S} = \str{V,X,\mathcal{A},\mathcal{F}}$ where $V$ is infinite,
    there exists a non-dictatorial social welfare function $\sigma$ for
    $\mathcal{S}$.
    
    A social welfare function $\sigma$ for $\mathcal{S}$ is
    \emph{$k$-non-dictatorial} if for all $s \in \Seq(V)$ such that
    $\lh{s} \leq k$, there exists $j$ and $x,y \in X$ such that for all
    $i < \lh{s}$, $x <_{j(s(i))} y$ and $y <_{\sigma(j)} x$.
    $\FPT^k$ is the statement obtained by replacing non-dictatoriality
    in $\FPT$ with $k$-non-dictatoriality for some fixed $k \geq 1$.
    
    $\sigma$ is \emph{finitely non-dictatorial} if for all $k \geq 1$, $\sigma$
    is $k$-non-dictatorial.
    $\FPT^{<\N}$ is the statement obtained by replacing
    non-dictatoriality in $\FPT$ with finite non-dictatoriality.
    
    $\sigma$ has the \emph{cofinite coalitions property} if for every profile
    $j \in \N$, if cofinitely many $v \in V$ are such that $x <_{j(v)} y$, then
    $x <_{\sigma(j)} y$.
    $\FPT^+$ is the statement obtained by replacing non-dictatoriality
    in $\FPT$ with the cofinite coalitions property.
\end{dfn}

One concern with the interpretation of Fishburn's possibility theorem has been
that the choice of ultrafilter seems arbitrary. When faced with an infinite set
with a complement of the same cardinality, there seems to be no reason to
consider one to genuinely constitute a majority rather than the other.
This is not the case for cofinite sets which, in an infinite society, clearly
constitute a majority. A social welfare function with the cofinite coalitions
property therefore satisfies a version of Condorcet consistency: if a majority
(a cofinite set) of voters prefer $x$ to $y$, then so does the social welfare
function.
Since an ultrafilter on a given algebra is non-principal exactly when it refines
the Fréchet filter, the cofinite coalitions property is also the strongest
non-dictatoriality property a social welfare function can have. We now show that
all non-dictatorial social welfare functions have this property.

\begin{lem}
    \label{lem:nd_equiv_cofcoal}
    The following is provable in $\RCA_0$.
    Suppose $\mathcal{S}$ is a countable society and $\sigma$ is a social
    welfare function for $\mathcal{S}$.
    Then the following conditions are equivalent.
    \begin{enumerate}
        \item $\sigma$ is non-dictatorial.
        \item $\sigma$ is $k$-non-dictatorial for some fixed $k \geq 1$.
        \item $\sigma$ is finitely non-dictatorial.
        \item $\sigma$ has the cofinite coalitions property.
    \end{enumerate}
\end{lem}

\begin{proof}
    The implications from 4 to 3, 3 to 2, and 2 to 1 are immediate.
    Working in $\RCA_0$, we show that 1 implies 4.
    Let $\mathcal{S} = \str{V,X,\mathcal{A},\mathcal{S}}$ be a countable society
    and let $\sigma$ be a non-dictatorial social welfare function for
    $\mathcal{S}$. By $\KS$ (theorem~\ref{thm:KS_RCA0}) the ultrafilter
    $\Uf_\sigma$ of (indexes of) $\sigma$-decisive coalitions exists and is
    non-principal. Moreover, $V$ is infinite by Arrow's theorem.
    
    Fix an arbitrary profile $f_m$ and two alternatives $x, y \in X$, and
    suppose that for some $k$, if $v \in V$ is such that $v \geq k$ then
    $x <_{m(v)} y$.
    By the closure of $\mathcal{A}$ under finite unions and relative complements
    there exists a $j$ such that $A_j = \{ v \in V : v \geq k \}$, which is
    cofinite since $V$ is infinite. Since $\Uf_\sigma$ is non-principal,
    $j \in \Uf_\sigma$ by part~4 of lemma~\ref{lem:uf_basic}. Therefore, $A_j$
    is $\sigma$-decisive and $x <_{\sigma(m)} y$.
\end{proof}

The following lemma~\ref{lem:cKS_RCA0} is a partial converse of the
Kirman--Sondermann theorem for countable societies---partial because for any
given ultrafilter $\Uf$ there may be distinct social welfare functions with
$\Uf$ as their set of decisive coalitions.
Various restrictions allow a one-to-one correspondence between ultrafilters and
social welfare functions to be recovered, for example by restricting to profiles
and social welfare functions which output linear orders as in 
\cite[theorem~6.1.3]{Taylor2005}, or by imposing a monotonicity condition as in
\cite{Armstrong1985}.

These restrictions are less interesting from a computability-theoretic point of
view, since the resulting bijective functionals between ultrafilters and social
welfare functions are themselves computable, while without these restrictions
there are social welfare functions $\sigma$ such that $\Uf_\sigma \Tlt \sigma$.
This can occur most strikingly when $\sigma$ is dictatorial, and hence
$\Uf_\sigma$ is computable (since to compute membership in $\Uf_\sigma$ one
simply needs to check for any given $A_i$ if $d \in A_i$, where $d$ is the
dictator). There will remain infinitely many profiles $f_i$ and alternatives
$x,y$ such that neither $\mu_<(i,x,y)$ nor $\mu_<(i,y,x)$ are in $\Uf_\sigma$.
Some of these gaps of indifference can be filled in by appealing to another,
non-principal and non-computable ultrafilter, resulting in a social welfare
function that is dictatorial but not computable. For details of this
construction see proposition~1 of \cite{Mihara1997}.

\begin{lem}
    \label{lem:cKS_RCA0}
    The following statement is provable in $\RCA_0$.
    Suppose $\mathcal{S} = \str{V,X,\mathcal{A},\mathcal{F}}$ is a countable
    society.
    If $\Uf$ is an ultrafilter on $\mathcal{A}$, then there exists a
    social welfare function $\sigma_\Uf$ for $\mathcal{S}$ with the
    following properties.
    \begin{enumerate}
        \item For all $i \in \N$, $i \in \Uf$ if and only if $A_i$ is
            $\sigma_\Uf$-decisive.
        \item The following conditions are equivalent:
            \begin{enumerate}
                \item $\Uf$ is non-principal,
                \item $\sigma_\Uf$ has the cofinite coalitions property.
            \end{enumerate}
    \end{enumerate}
\end{lem}

\begin{proof}
    Working in $\RCA_0$, let $\mathcal{S} = \str{V,X,\mathcal{A},\mathcal{F}}$
    be a countable society and $\Uf \subseteq \N$ be an ultrafilter on
    $\mathcal{A}$.
    
    Let $\varphi(n,R)$ be the following $\Sigma^0_0$ formula in the displayed
    free variables.
    \begin{equation*}
        \varphi(n,R) \equiv (\forall{x,y} \in X)(
            (x,y) \in R \leftrightarrow \mu(n,x,y) \in \Uf
        ).
    \end{equation*}
    Note that here we are considering $R$ as a natural number coding a finite
    set.
    Let $b$ code the finite set $X \times X$. Since our coding of finite sets by
    natural numbers is monotonic, $b \geq R'$ for all $R' \in W$.
    By $\Sigma^0_1$ induction, for all $n$ there exists $R \leq b$ such that
    $\varphi(R,n)$. This is just an application of comprehension for codes of
    finite sets; for details see e.g.\ \cite{HajekPudlak1993}.
    
    
    We show that $R$ is a weak order.
    To show strong connectedness, let $x,y \in X$ be arbitrary.
    If $x = y$ then since $x \lesssim_{n(v)} x$ for all $n \in \N$ and
    $v \in V$, we have that $A_{\mu(n,x,x)} = V$, so $\mu(n,x,x) \in \Uf$ by
    non-emptiness and thus $(x,x) \in R$.
    Suppose instead that $x \neq y$, let $i = \mu(n,x,y)$ and let $j$ be such
    that $A_j = A_i^c$. Since $\Uf$ is an ultrafilter, by maximality either
    $i \in \Uf$ or $j \in \Uf$. If $i \in \Uf$ then $(x,y) \in R$, so assume the
    latter. $A_j = A_{\mu_<(n,y,x)} \subseteq A_{\mu(n,y,x)}$, so
    $\mu(n,y,x) \in \Uf$ by upwards closure, establishing that $(y,x) \in R$.
    
    For transitivity, suppose $(x,y) \in R$ and $(y,z) \in R$, so
    $\mu(n,x,y) \in \Uf$ and $\mu(n,y,z) \in \Uf$. Let $j$ be such that
    $A_j = A_{\mu(n,x,y)} \cap A_{\mu(n,y,z)}$, so $j \in \Uf$ by closure under
    intersections. Then $x \lesssim_{n[A_j]} y$ and $y \lesssim_{n[A_j]} z$, so
    by transitivity we have that
    $x \lesssim_{n[A_j]} z$. Thus, $A_j \subseteq A_{\mu(n,x,z)}$ and
    $\mu(n,x,z) \in \Uf$ by upwards closure.
    
    
    This lets us define $\sigma \subseteq \N$ by
    \begin{align*}
        (n,R) \in \sigma
            \leftrightarrow
            R = \min R' \text{ such that } \varphi(R',n).
    \end{align*}
    Since $W$ is finite, the use of minimisation is bounded and so the
    definition of $\sigma$ is $\Sigma^0_0$ in the parameters $\mu_<$ and $\Uf$,
    meaning that $\sigma$ exists by recursive comprehension.
    By the claim, $\sigma \subseteq \N \times W$ and for all $n \in \N$ there
    exists $R \in W$ such that $(n,R) \in \sigma$. Thus, since minimisation is a
    function, so is $\sigma$, i.e.\ $\sigma : \N \to W$.
    
    
    We now show that $m \in \Uf$ if and only if $A_m$ is $\sigma$-decisive.
    For the forwards direction, suppose $m \in \Uf$ and $x,y \in X$ and
    $n \in \N$ are such that $x <_{n[A_m]} y$. By this hypothesis,
    $A_m \subseteq A_{\mu(n,x,y)}$, so $\mu(n,x,y) \in \Uf$ by upwards closure.
    $A_{\mu(n,x,y)} = A_{\mu_<(n,x,y)} \cup A_{\mu_\sim(n,x,y)}$, and thus
    either $\mu_<(n,x,y) \in \Uf$ or $\mu_\sim(n,x,y) \in \Uf$ by part~2 of
    lemma~\ref{lem:uf_basic}. Suppose the latter. By hypothesis,
    $A_{\mu_\sim(n,x,y)} \cap A_m = \emptyset$, and since $\Uf$ is closed under
    intersections it would have to contain an index for $\emptyset$,
    contradicting properness. So $\mu_<(n,x,y) \in \Uf$,
    $\mu_\sim(n,x,y) \not\in \Uf$, and $\mu_<(n,y,x) \not\in \Uf$, which
    establishes that $x <_{\sigma(n)} y$ by the definition of $\sigma$.
    For the reverse direction, suppose $A_m$ is $\sigma$-decisive and let
    $x,y \in X$ be arbitrary. By quasi-partition embedding there exists $f_k$
    such that $x <_{k(v)} y$ if $v \in A_m$, and $y <_{k(v)} x$ if
    $v \in A_m^c$. By $\sigma$-decisiveness, $x <_{\sigma(k)} y$, so
    $\mu(k,x,y) \in \Uf$. $A_m = A_{\mu(k,x,y)}$, so $m \in \Uf$ by upwards
    closure.
    
    
    To show that $\sigma$ satisfies unanimity, let $x,y \in X$ and $f_n$ be
    arbitrary, and suppose that $x <_{n[V]} y$. Because $A_{\mu(n,x,y)} = V$
    by uniform $\mathcal{A}$-measurability, it follows by the non-emptiness
    condition for $\Uf$ that $\mu(n,x,y) \in \Uf$. Moreover, we also have that
    $A_{\mu_<(n,y,x)} = A_{\mu_\sim(n,x,y)} = \emptyset$, so
    $\mu_<(n,y,x) \not\in \Uf$ and $\mu_\sim(n,x,y) \not\in \Uf$. It follows
    that by the construction of $\sigma$, $x <_{\sigma(n)} y$.
    
    
    To show that $\sigma$ satisfies independence, let $x,y \in X$ and suppose
    $f_i = f_j$ on $\{ x, y \}$. $A_{\mu(i,x,y)} = A_{\mu(j,x,y)}$ by uniform
    $\mathcal{A}$-measurability. Upwards closure of $\Uf$ under $\subseteq$ then
    gives us that $\mu(i,x,y) \in \Uf \leftrightarrow \mu(j,x,y)$. By the
    construction of $\sigma$, $x \lesssim_{\sigma(i)} y \leftrightarrow
    x \lesssim_{\sigma(j)} y$ as desired.
    
    
    Finally we prove that $\Uf$ is non-principal if and only if $\sigma$ has the
    cofinite coalitions property.
    For the forwards direction, suppose $\Uf$ is non-principal and let $A_i$
    be cofinite, so $i \in \Uf$ by part~4 of lemma~\ref{lem:uf_basic}. Suppose
    that $x <_{k[A_i]} y$ for some $x,y \in X$ and $k \in \N$. Since
    $i \in \Uf$, $A_i$ is $\sigma$-decisive, and so $x <_{\sigma(k)} y$.
    For the backwards direction, suppose $\sigma$ has the cofinite coalitions
    property and let $A_i$ be cofinite. By quasi-partition embedding, let $j$ be
    such that $A_{\mu(j,x,y)} = \{ v : x <_{j(v)} y \} = A_i$.
    $x <_{\sigma(j)} y$ by the cofinite coalitions property since $A_i$ is
    cofinite, so $\mu(j,x,y) \in \Uf$, and hence $i \in \Uf$ by upwards closure
    under $\subseteq$. Since $i$ was arbitrary, $\Uf$ is non-principal by part~4
    of lemma~\ref{lem:uf_basic}.
\end{proof}

\begin{thm}
    \label{thm:FPT_equiv_ACA0}
    The following are equivalent over $\RCA_0$.
    \begin{enumerate}
        \item $\FPT$.
        \item $\FPT^k$ for any $k \geq 1$.
        \item $\FPT^{<\N}$.
        \item $\FPT^+$.
        \item Arithmetical comprehension.
    \end{enumerate}
\end{thm}

\begin{lem}
    \label{lem:partition_embedding}
    The following is provable in $\RCA_0$.
    Suppose $V \subseteq \N$ is nonempty and $X \subseteq \N$ is finite with
    $\lh{X} \geq 3$ and $\mathcal{A} = \str{ A_i : i \in \N }$ is a countable
    algebra over $V$. Then there exists a sequence $\mathcal{F} =
    \str{ f_i : i \in \N }$ of profiles over $V,X$ such that $\mathcal{F}$ is
    uniformly $\mathcal{A}$-measurable and $\mathcal{A}$ is quasi-partition
    embedded into $\mathcal{F}$.
\end{lem}

\begin{proof}
    We first apply lemma~\ref{lem:boolean_embedding} to replace $\mathcal{A}$
    with an extensionally equivalent algebra $\mathcal{A}'$ in which we can
    uniformly compute boolean combinations via a boolean embedding. We abuse
    notation in the remainder of this proof by referring to $\mathcal{A}$ rather
    than $\mathcal{A}'$.
    
    The infinite set $\mathrm{Perm}(W) \times \mathrm{QPart}(\lh{W})$ exists by
    recursive comprehension, and by primitive recursion there exists a function
    $\mathit{en} : \N \to \mathrm{Perm}(W) \times \mathrm{QPart}(\lh{W})$
    enumerating it.
    Let $\theta(n,v,w)$ be a $\Sigma^0_0$ formula which says that
    $\mathit{en}(n) = (p,s)$ and either there exists a unique $j<\lh{s}-1$ such
    that $v \in A_{s(j)}$ and $w = p(i)$, or there exists no such unique $j$ and
    $w = p(\lh{s}-1)$.
    The set $\mathcal{F} = \{ (n,v,w) : \theta(n,v,w) \}$ exists by recursive
    comprehension and codes a sequence of profiles $\str{ f_i : i \in \N }$.
    
    We now show that $e = \mathit{en}^{-1}$ is a quasi-partition embedding of
    $\mathcal{A}$ into $\mathcal{F}$. Let $p$ be a permutation of $W$, $s$ a
    quasi-partition, and $k = e(p,s)$. Suppose $v \in V$. We reason by cases.
    \begin{enumerate}
        \item Suppose there exists a unique $j<\lh{s}-1$ such that
            $v \in A_{s(j)}$. Then $(k,v,p(j)) \in \mathcal{F}$ by the
                construction of $\mathcal{F}$, i.e.\ $f_k(v) = p(j)$.
        \item Now suppose there is no such $j$. Then
            $(k,v,p(\lh{s}-1)) \in \mathcal{F}$ by the construction of
            $\mathcal{F}$, i.e.\ $f_k(v) = p(\lh{s}-1)$.
    \end{enumerate}
    
    Finally we show that $\mathcal{F}$ is uniformly $\mathcal{A}$-measurable.
    Fix $x,y \in X$ and a profile $f_n$. By the construction of $\mathcal{F}$,
    $\mathit{en}(n) = (s,p)$ for some quasi-partition $s$ and permutation $p$ of
    $W$.
    For all $j < \lh{s}$, let $t^j$ be a boolean formation sequence for the set
    \begin{equation*}
        A_{s(j)} \setminus
            \bigcup_{i < \lh{s}-1} \left\{\begin{array}{ll}
                A_{s(i)}  & \text{if } i \neq j, \\
                \emptyset & \text{otherwise}.
            \end{array}\right.
    \end{equation*}
    and given boolean formation sequences $t_1$ and $t_2$, let
    \begin{equation*}
        \begin{split}
        u(t_1,t_2) = t_1 \concat t_2 \concat \str{&
            (1,\lh{t_1}-1,\lh{t_1}-1), \\
            & (1,\lh{t_1}+\lh{t_2}-1,\lh{t_1}+\lh{t_2}-1), \\
            & (2,\lh{t_1}+\lh{t_2},\lh{t_1}+\lh{t_2}+1)
        }.
        \end{split}
    \end{equation*}
    
    Let $h_0(s,p,x,y) = \str{}$ and
    \begin{equation*}
        h_r(t,m,s,p,x,y) = \begin{cases}
            u(t,t^j) & \text{if } x <_{p(s(m-1))} y,  \\
            t        & \text{otherwise}.
        \end{cases}
    \end{equation*}
    Let $h$ be the function defined by primitive recursion from $h_0$ and $h_r$.
    Define $\mu : \N \times X \times X \to \N$ by
    $\mu(n,x,y) = e^*(h(\lh{s},s,p,x,y))$,
    where $e^* : \mathrm{BFS} \to \N$ is a boolean embedding of $\mathcal{A}$
    into itself (this exists by lemma~\ref{lem:boolean_embedding}).
    We can then verify by $\Sigma^0_0$ induction that $\mathcal{F}$ is uniformly
    $\mathcal{A}$-measurable via $\mu$.
\end{proof}

\begin{proof}[Proof of theorem~\ref{thm:FPT_equiv_ACA0}]
    Statements 1, 2, 3, and 4 are equivalent by
    lemma~\ref{lem:nd_equiv_cofcoal}.
    To complete the proof it suffices to show that 5 implies 4 and 1 implies 5.
    To show that 5 implies 4, we work in $\ACA_0$ and suppose that $\mathcal{S}
    = \str{V,X,\mathcal{A},\mathcal{F}}$ is a countable society and that $V$ is
    infinite.
    By lemma~\ref{lem:uf_exist_equiv_ACA0}, there exists a non-principal
    ultrafilter $\Uf$ on $\mathcal{A}$, and hence by lemma~\ref{lem:cKS_RCA0}
    there exists a social welfare function $\sigma_\mathcal{U}$ for
    $\mathcal{S}$ with the cofinite coalitions property.
    
    Finally we show that 1 implies 5. Working in $\RCA_0$, let $V \subseteq \N$
    be infinite and let $\mathcal{A}$ be a countable atomic algebra over $V$.
    Fix $X = \{ x, y, z \}$.
    By lemma~\ref{lem:partition_embedding} there exists a countable sequence of
    profiles $\mathcal{F}$ over $V,X$ such that $\mathcal{A}$ is quasi-partition
    embedded into $\mathcal{F}$ and $\mathcal{F}$ is uniformly
    $\mathcal{A}$-measurable.
    $\mathcal{S} = \str{V,X,\mathcal{A},\mathcal{F}}$ is thus a countably
    infinite society, and so by $\FPT$ there exists a non-dictatorial
    social welfare function $\sigma$ for $\mathcal{S}$.
    By $\KS$ (theorem~\ref{thm:KS_RCA0}), there exists an ultrafilter
    $\Uf_\sigma$ on $\mathcal{A}$ which is non-principal since $\sigma$ is
    non-dictatorial. Since $\mathcal{A}$ is an arbitrary infinite atomic
    algebra, this implies arithmetical comprehension by
    lemma~\ref{lem:uf_exist_equiv_ACA0}.
\end{proof}

We conclude this section with a few remarks on the computability-theoretic
status of $\FPT$.
Early work in effectivising social choice theory emphasised the
non-computability of non-dictatorial social welfare functions, and thus an
extension of Arrow's theorem from finite sets to computable sets.
For example, Mihara~\cite{Mihara1997} showed that when $V = \N$, $\mathcal{A} =
\REC$, and $\mathcal{F}$ consists of all $\mathcal{A}$-measurable profiles, any
non-dictatorial social welfare function for this society is non-computable.
In the present setting this is not automatic: there are countable societies
with computable non-dictatorial social welfare functions. The natural minimal
example of this is provided by a society based on a computable presentation of
the finite--cofinite algebra. There is a single non-principal ultrafilter on
this algebra, and both it and the non-dictatorial social welfare function
derived from it via the construction in lemma~\ref{lem:cKS_RCA0} are
computable.

On the other hand, there are recursive counterexamples to Fishburn's possibility
theorem far less complex than the societies considered by Lewis~\cite{Lewis1988}
or Mihara~\cite{Mihara1997} which we discussed in \S\ref{sec:intro}.
The following argument is based on Kirby's proof of the reverse direction of
lemma~\ref{lem:uf_exist_equiv_ACA0} \cite[theorem~1.10]{Kirby1984}.
Let $h : \N \to \N$ be a computable enumeration of the halting problem $\zeroj$.
Using lemma~\ref{lem:boolean_embedding} we computably embed a sequence of sets
$B = \str{ B_i : i \in \N }$ into a countable atomic algebra $\mathcal{A}$,
where $B$ is defined by
\begin{equation*}
    B =      \{ (2n,   v) : (\exists{m < v})(h(m) = n) \}
        \cup \{ (2n+1, n) : n \in \N \}.
\end{equation*}
By lemma~\ref{lem:partition_embedding} there exists a countable society
$\mathcal{S} = \str{\N,3,\mathcal{A},\mathcal{F}}$. We can then construct a
primitive recursive function $g : \N \to \N$ that computes the indexes of a
family of profiles such that $x <_{g(n)(v)} y$ if $v \in B_{2n}$, and
$y <_{g(n)(v)} x$ otherwise.
If $\sigma$ is any non-dictatorial social welfare function for $\mathcal{S}$,
then $\zeroj \Tleq \sigma$, since $\zeroj = \ran(h)$ is $\Sigma^0_0$ definable
in the parameter $\sigma$ by the formula $\varphi(n) \equiv
x <_{\sigma(g(n))} y$.
There will only exist a $v$ such that $v \in B_{2n}$ if $n \in \ran(h)$, but
when there is, the cofinite coalitions property ensures that $x <_{f(v)} y$.
$\mathcal{S}$ is thus a computable society all of whose non-dictatorial social
welfare functions compute $\zeroj$.
Nevertheless, this non-computability result is `easy' since it only requires
coding a single jump. A natural question is thus whether we can obtain more
precise degree-theoretic information about the complexity of non-dictatorial
social welfare functions.


\section{Conclusion and further work}
\label{sec:conclusion}

The results presented in this paper initiate the reverse mathematics of social
choice theory. In doing so, they demonstrate both the suitability of reverse
mathematics as a framework in which to assess the effectivity of theorems from
social choice theory, and the fruitfulness of social choice theory as a source
for reverse mathematical results.
It is straightforward within the present setting to define additional types of
collective choice rules for countable societies, allowing further theorems like
Sen's liberal paradox \cite{Sen1970} or the Gibbard--Satterthwaite theorem
\cite{Gibbard1973, Satterthwaite1975} to be formalised in $\lang_2$, and their
reverse mathematical status to be investigated.
The latter theorem, which concerns strategic voting and the manipulability of
elections, is a classical impossibility result from the 1970s. Like Arrow's
theorem, it has a corresponding possibility theorem when the set of voters is
infinite \cite{PaznerWesley1977}.
Finally, the equivalence between $\FPT$ and arithmetical comprehension shows
that the existence of non-computable sets is essential to proving the existence
of non-dictatorial social welfare functions. On the one hand, this is a far
weaker notion of non-constructivity than that measured by equivalences to choice
principles over $\ZF$. On the other, it shows that we cannot in general hope for
computable rules for social decision-making in countably infinite societies,
even for countable societies whose coalitions do not include all computable sets
of voters.


\subsection*{Acknowledgements}
\label{sec:acknowledgements}
%
%
%
%
%
%

I would like to thanks the referees for their helpful suggestions, as well as
many colleagues who provided feedback on earlier versions of this paper:
Marianna Antonutti Marfori,
Peter Cholak,
Walter Dean,
Damir Dzhafarov,
Peter Hammond,
Hannes Leitgeb,
Toby Meadows,
H.\ Reiju Mihara,
Carl Mummert,
Arianna Novaro, and
Marcus Pivato.


\bibliographystyle{plain}
\bibliography{research}

\end{document}